\documentclass[11pt]{amsart}

\usepackage[
paper=a4paper,
text={147mm,230mm},centering
]{geometry}
\usepackage{dsfont}
\iftrue
\makeatletter
\def\@settitle{%
  \vspace*{-20pt}
  \begin{flushleft}%
    \baselineskip14\p@\relax
    \normalfont\bfseries\LARGE
    \@title
  \end{flushleft}%
}
\def\@setauthors{%
  \begingroup
  \def\thanks{\protect\thanks@warning}%
  \trivlist
  \large \@topsep30\p@\relax
  \advance\@topsep by -\baselineskip
  \item\relax
  \author@andify\authors
  \def\\{\protect\linebreak}%
  \authors
  \ifx\@empty\contribs
  \else
    ,\penalty-3 \space \@setcontribs
    \@closetoccontribs
  \fi
  \normalfont
  \@setaddresses
  \endtrivlist
  \endgroup
}
\def\@setaddresses{\par
  \nobreak \begingroup\raggedright
  \small
  \def\author##1{\nobreak\addvspace\smallskipamount}%
  \def\\{\unskip, \ignorespaces}%
  \interlinepenalty\@M
  \def\address##1##2{\begingroup
    \par\addvspace\bigskipamount\noindent
    \@ifnotempty{##1}{(\ignorespaces##1\unskip) }%
    {\ignorespaces##2}\par\endgroup}%
  \def\curraddr##1##2{\begingroup
    \@ifnotempty{##2}{\nobreak\noindent\curraddrname
      \@ifnotempty{##1}{, \ignorespaces##1\unskip}\/:\space
      ##2\par}\endgroup}%
  \def\email##1##2{\begingroup
    \@ifnotempty{##2}{\smallskip\nobreak\noindent E-mail address%
      \@ifnotempty{##1}{, \ignorespaces##1\unskip}\/:\space
      \ttfamily##2\par}\endgroup}%
  \def\urladdr##1##2{\begingroup
    \def~{\char`\~}%
    \@ifnotempty{##2}{\nobreak\noindent\urladdrname
      \@ifnotempty{##1}{, \ignorespaces##1\unskip}\/:\space
      \ttfamily##2\par}\endgroup}%
  \addresses
  \endgroup
  \global\let\addresses=\@empty
}
\def\@setabstracta{%
    \ifvoid\abstractbox
  \else
    \skip@25\p@ \advance\skip@-\lastskip
    \advance\skip@-\baselineskip \vskip\skip@
    \box\abstractbox
    \prevdepth\z@ 
    \vskip-10pt
  \fi
}
\renewenvironment{abstract}{%
  \ifx\maketitle\relax
    \ClassWarning{\@classname}{Abstract should precede
      \protect\maketitle\space in AMS document classes; reported}%
  \fi
  \global\setbox\abstractbox=\vtop \bgroup
    \normalfont\small
    \list{}{\labelwidth\z@
      \leftmargin0pc \rightmargin\leftmargin
      \listparindent\normalparindent \itemindent\z@
      \parsep\z@ \@plus\p@
      
    }%
    \item[\hskip\labelsep\bfseries\abstractname.]%
}{%
  \endlist\egroup
  \ifx\@setabstract\relax \@setabstracta \fi
}
\def\section{\@startsection{section}{1}%
  \z@{-1.2\linespacing\@plus-.5\linespacing}{.8\linespacing}%
  {\normalfont\bfseries\large}}
\def\subsection{\@startsection{subsection}{2}%
  \z@{-.8\linespacing\@plus-.3\linespacing}{.3\linespacing\@plus.2\linespacing}%
  {\normalfont\bfseries}}
\def\subsubsection{\@startsection{subsubsection}{3}%
  \z@{.7\linespacing\@plus.1\linespacing}{-1.5ex}%
  {\normalfont\itshape}}
\def\@secnumfont{\bfseries}
\makeatother
\fi 

\usepackage{amssymb}
\usepackage[all]{xy}
\usepackage{graphicx,color,float}
\usepackage[bookmarks]{hyperref}
\usepackage{pinlabel}
\usepackage[textwidth=1.3in,color=yellow]{todonotes}
\usepackage{amsmath}
\usepackage{pb-diagram,pb-xy}
\usepackage{url}
\textwidth=\paperwidth \advance\textwidth by-2\oddsidemargin
\advance\oddsidemargin by-1in
\evensidemargin=\oddsidemargin
\calclayout

\def\Z{\mathbb{Z}}

\def\Q{\mathbb{Q}}

\def\d{\partial}

\def\Ker{\operatorname{Ker}}
\def\Coker{\operatorname{Coker}}

\def\+{\oplus}
\def\hat{\widehat}

\theoremstyle{plain}
\newtheorem{theorem}{Theorem}[section]
\newtheorem{proposition}[theorem]{Proposition}
\newtheorem{corollary}[theorem]{Corollary}
\newtheorem{lemma}[theorem]{Lemma}
\theoremstyle{definition}
\newtheorem{definition}[theorem]{Definition}
\newtheorem{example}[theorem]{Example}
\newtheorem{remark}[theorem]{Remark}
\newtheorem*{question}{Question}

\def\to{\mathchoice{\longrightarrow}{\rightarrow}{\rightarrow}{\rightarrow}}
\makeatletter
\newcommand{\shortxra}[2][]{\ext@arrow 0359\rightarrowfill@{#1}{#2}}
\def\longrightarrowfill@{\arrowfill@\relbar\relbar\longrightarrow}
\newcommand{\longxra}[2][]{\ext@arrow 0359\longrightarrowfill@{#1}{#2}}
\renewcommand{\xrightarrow}[2][]{\mathchoice{\longxra[#1]{#2}}%
  {\shortxra[#1]{#2}}{\shortxra[#1]{#2}}{\shortxra[#1]{#2}}}
\makeatother
\begin{document}

\title [On rational sliceness of Miyazaki's fibered, $-$amphicheiral knots]
{On rational sliceness of Miyazaki's fibered, $-$amphicheiral knots}

\author{Min Hoon Kim}
\address{
  School of Mathematics\\
  Korea Institute for Advanced Study \\
  Seoul 130--722\\
  Republic of Korea
}
\email{kminhoon@gmail.com}
\author{Zhongtao Wu}
\address{
Department of Mathematics\\
Lady Shaw Building\\
The Chinese University of Hong Kong\\
Shatin\\
Hong Kong
}
\email{ztwu@math.cuhk.edu.hk}



\maketitle
\begin{abstract}We prove that certain fibered, $-$amphicheiral knots are rationally slice.  Moreover, we show that the concordance invariants $\nu^+$ and $\Upsilon(t)$ from Heegaard Floer homology vanish for a class of knots that includes rationally slice knots.
\end{abstract}
\section{Introduction}
Recall that a knot $K\subset S^3$ is called \emph{slice} if it bounds an embedded disk $D$ in $D^4$, and it is called \emph{ribbon} if it bounds an immersed disk in $S^3$ with only ribbon singularities.\footnote{The present paper considers smooth category unless otherwise specified.} One easily sees that every ribbon knot is a slice knot.  An outstanding open problem, posed by Fox and known as the \emph{slice-ribbon conjecture}, asks if the converse is true. As an attempt to approach the problem, Casson and Gordon introduced the notion of \emph{homotopy ribbon} in \cite{Casson-Gordon:1983-1}. A knot $K$ is homotopy ribbon if it bounds an embedded disk $D$ in a homotopy 4-ball $V$ so that the inclusion induced map $$\pi_1(S^3-K)\rightarrow \pi_1(V-D)$$ is surjective. Since every ribbon knot is homotopy ribbon, the slice-ribbon problem can be divided into two parts, namely whether every slice knot is homotopy ribbon, and whether every homotopy ribbon knot is ribbon \cite[Problem 4.22]{Kirby:problem-list-1995-edition}.

In \cite[Theorem 5.1]{Casson-Gordon:1983-1}, it is proved that a fibered knot is homotopy ribbon if and only if the monodromy of its fiber extends over the handlebody. Hinging on this idea, Miyazaki showed that the connected sum of iterated torus knots $T_{2,3;2,13}\#T_{2,15}\#-T_{2,3;2,15}\#-T_{2,13}$ is algebraically slice but not homotopy ribbon \cite[Example 1]{Miyazaki:1994-1}; such a knot, according to Hedden, Kirk and Livingston \cite{Hedden-Kirk-Livingston:2012-1}, is not even topologically slice.  Our paper will center around Miyazaki's another example of algebraically slice, non-homotopy-ribbon fibered knots, whose construction is based on the following specific family of knots that we will refer to as \emph{Miyazaki knots} for the rest of the paper.
\begin{definition}\label{definition:Miyazaki}
A Miyazaki knot $K$ is a fibered, $-$amphicheiral knot with irreducible Alexander polynomial $\Delta_K(t)$.
\end{definition}

\begin{proposition}[Example 2, \cite{Miyazaki:1994-1}]\label{example:Miyazaki}

For a Miyazaki knot $K$, the cable knot $K_{2n,1}$ is algebraically slice but not homotopy ribbon.
\end{proposition}

In view of the slice-ribbon conjecture, one asks the following open question.
\begin{question}[Kawauchi] Is $K_{2n,1}$ slice when $K$ is Miyazaki?  In particular, is $(4_1)_{2,1}$ a slice knot, where $4_1$ denotes the figure-eight knot?\footnote{Note that $4_1$ is a Miyazaki knot. The special case of $(4_1)_{2,1}$ is originally due to Kawauchi \cite{Kawauchi:1980-2}.}
\end{question}
Note that if a certain $K_{2n,1}$ is slice, then this knot will be a counterexample to the slice-ribbon conjecture.  To this end, significant progress has been made on establishing the sliceness of those knots in the following weaker sense.

\begin{definition}\label{Def:RS}For a subring $R\subset \Q$, a knot $K\subset S^3$ is called \emph{$R$-slice} if there exists an embedded disk $D$ in an $R$-homology 4-ball $V$ such that $\partial(V,D)=(S^3, K)$. When $R=\Q$, we say $K$ is \emph{rationally slice}. We are mainly interested in the case that $R=\Q$ or $R=\Z[\frac{1}{n}]$, the subring of $\Q$ generated by $\frac{1}{n}$.
\end{definition}
\begin{remark}Our definition of rationally slice knots is weaker than the one used in \cite{Kawauchi:2009-1} by Kawauchi.  To avoid confusion, we call a knot $K$ in $S^3$ strongly rationally slice if $(S^3,K)=\partial (V,D)$ where $D$ is a smoothly embedded disk in a $\Q$-homology 4-ball $V$ such that the following inclusion induced map is an isomorphism:
\[H_1(S^3-\nu(K);\Z)/\textrm{torsion}\xrightarrow{\cong} H_1(V-\nu(D);\Z)/\textrm{torsion}.\]
We will not use this definition but we remark that if a knot $K$ in $S^3$ is strongly rationally slice, then $K$ is algebraically slice. By \cite[Theorem 4.16]{Cha:2007-1}, $4_1$ is rationally slice in our sense but $4_1$ is \emph{not} strongly rationally slice since $4_1$ is not algebraically slice.
\end{remark}

In \cite{Kawauchi:1979-1,Kawauchi:2009-1}, Kawauchi showed that the Miyazaki knot $K$ is rationally slice if $K$ is hyperbolic (see Lemma \ref{lemma:Kawauchi}).  Our first main theorem extends this result to all Miyazaki knots.

\begin{theorem}\label{main1}
If $K$ is Miyazaki, then $K$ is $\Z[\frac{1}{2}]$-slice. In particular, $K$ is rationally slice.
\end{theorem}
\begin{remark}\label{remark:generalization}We compare Theorem \ref{main1} with work of Cochran, Davis, and Ray \cite{Cochran-Davis-Ray:2014-1} which generalizes work of Ray \cite{Ray:2013-1}. By \cite[Corollary 5.3]{Cochran-Davis-Ray:2014-1}, if $K_{2n,1}$ is slice, then $K$ is $\Z[\frac{1}{2n}]$-slice. In particular, if there is a counterexample to the slice-ribbon conjecture given by $K_{2n,1}$ for some Miyazaki knot $K$, then $K$ will be $\Z[\frac{1}{2n}]$-slice. Therefore, Theorem \ref{main1} can be viewed as supporting evidence of the existence of such a counterexample.  (A universal coefficient theorem argument gives that $\Z[\frac{1}{2}]$-slice implies $\Z[\frac{1}{2n}]$-slice.)

We remark that this cannot be improved to the $\Z$-sliceness. Indeed, it is easy to see that no Miyazaki knot is $\Z$-slice. It is known that every $\Z$-slice knot is algebraically slice and hence its Alexander polynomial is reducible. However, the Alexander polynomial of a Miyazaki knot is irreducible. We also remark that the special case of Theorem \ref{main1}  when $K=4_1$ is proved implicitly in \cite[Theorem 4.16]{Cha:2007-1} using a Kirby calculus argument.
\end{remark}

On the contrary, one may attempt to prove the non-sliceness of cables of Miyazaki knots from the perspective that slice-ribbon conjecture could be true.  This leads us to look at their concordance invariants coming from Heegaard Floer theory. We would be primarily working with $\nu^+$ invariant \cite{Hom-Wu:2014-1}, as it gives the strongest obstruction on sliceness among several closely related invariants, including $\tau$, $\epsilon$, and $\Upsilon(t)$ invariants.  For a detailed description and comparison of concordance invariants from Heegaard Floer theory, we refer the readers to Hom's survey article \cite{Hom:2015-2}.  For the purpose of the present paper, remember that a slice knot $K$ has $\nu^+(K)=0$.  The cabling formula from \cite{Wu:2015-1} implies that $\nu^+(K_{2n,1})=0$ if and only if $\nu^+(K)=0$.  Hence, provided that we can find a Miyazaki knot $K$ with non-vanishing $\nu^+$ invariant, the knots $K_{2n,1}$ would be non-slice since they also have non-vanishing $\nu^+$ invariants.  This would give a negative answer to Kawauchi's question.

Also inspired by recent work of Cha and the first author \cite{Cha-Kim:2015-1}, we like to determine whether the collection of $\nu^+$ invariants of satellites detects slice knots.  More precisely, if $\nu^+(P(K))=\nu^+(P(U))$ holds for all satellite patterns $P\subset S^1\times D^2$, does it imply that $K$ is slice?  Our second main theorem addresses this question.

\begin{theorem}\label{main2}
If $K$ is a $\Q$-homology \textup{0}-bipolar knot, then
 \begin{enumerate}
\item $\nu^+(P(K)\#-P(U))=\nu^+(-P(K)\#P(U))=0$ for all patterns $P$.
\item $\nu^+(P(K))=\nu^+(P(U))$ for all patterns $P$.
\end{enumerate}
\end{theorem}

Here, the $\Q$-homology 0-bipolar knot is a $\Q$-homology analog of 0-bipolar knots whose precise definition will be given in Section 3.  As noted in \cite{Cha-Kim:2015-1}, there are $\Q$-homology 0-bipolar knots which generate a subgroup isomorphic to $\Z^\infty \oplus \Z_2^\infty$ in either the smooth or topological knot concordance groups.  In particular, all rationally slice knots are $\Q$-homology 0-bipolar, so their sliceness cannot be detected by the $\nu^+$ invariant by Theorem \ref{main2}.  Observe also that Theorem \ref{main1} says Miyazaki knots are rationally slice.  It follows that Miyazaki knots and their $(2m, 1)$-cable all have vanishing $\nu^+$ invariants.


In \cite{Cha:2007-1}, the rational knot concordance group $\mathcal{C}_\Q$ is defined as the set of rational concordance classes of knots in rational homology 3-spheres. There is a natural, induced inclusion map $\mathcal{C}\to \mathcal{C}_\Q$ from the smooth concordance group to the rational concordance group. In \cite[Problem 1.11]{Usher:2011-1}, Cha asked about the structure of the kernel and the cokernel of the map $\mathcal{C}\to \mathcal{C}_\Q$. In \cite[Theorem 1.4]{Cha:2007-1}, it is shown that $\Ker(\mathcal{C}\to \mathcal{C}_\Q)$ contains a $\Z_2^\infty$-subgroup and $\Coker(\mathcal{C}\to \mathcal{C}_\Q)$ contains a $\Z^\infty\oplus \Z_2^\infty\oplus \Z_4^\infty$-direct summand. As far as the authors know, this is the only known fact about $\mathcal{C}\to \mathcal{C}_\Q$, and the existence of a $\Z^\infty$-subgroup in $\Ker(\mathcal{C}\to \mathcal{C}_\Q)$ is an intriguing open problem. As a corollary of our main theorem, we resolve this problem modulo the slice-ribbon conjecture.
\begin{corollary}\label{corollary:slice-ribbon}If the slice-ribbon conjecture is true, then $\{K_{2n,1}\}_{n=1}^\infty$ generate a $\Z^\infty$-subgroup of $\Ker(\mathcal{C}\to \mathcal{C}_\Q)$ for any Miyazaki knot $K$.
\end{corollary}
\begin{remark}If we consider, instead, $\mathcal{C}^\textrm{top}$ and $\mathcal{C}_\Q^\textrm{top}$, the topological knot concordance group and the topological rational knot concordance group, respectively, then the same proof works for an analogous statement: If the topologically slice-homotopy ribbon conjecture is true, then $\{K_{2n,1}\}_{n=1}^\infty$ generate a $\Z^\infty$-subgroup of $\Ker(\mathcal{C}^\textrm{top}\to \mathcal{C}_\Q^\textrm{top})$ for any Miyazaki knot $K$. For a survey and the state of the art for the topological slice-homotopy ribbon conjecture, see \cite{Cha-Powell:2016-1}.

It is interesting to compare Corollary \ref{corollary:slice-ribbon} with the recent result of Abe-Tagami \cite[Lemma~3.1]{Abe-Tagami:2015-1} which is based on the work of Baker \cite{Baker:2016-1}: If the slice-ribbon conjecture is true, then the set of tight, prime fibered knots in $S^3$ are linearly independent in $\mathcal{C}$.
\end{remark}

The remaining part of this paper is organized as follows.  In Section 2, we discuss the amphicheirality, or more generally, the symmetry of knots and links.  We will relate the symmetry of a satellite knot with the symmetry of its companion and pattern, and also relate the symmetry of a knot in a solid torus with the symmetry of its associated link.  These relationship will be crucial to the understanding of satellite Miyazaki knots.  In Section 3, we show that every Miyazaki knot is either hyperbolic or a satellite knot with companion that is also Miyazaki and of smaller genus.  Subsequently, we apply an inductive argument to prove the main result, Theorem \ref{main1}, on rational sliceness of Miyazaki knots.  We then exhibit an infinite family of satellite Miyazaki knots in Section 4, and make a digression to observe that all known examples are strongly $-$amphicheiral.  Finally, in Section 5, we take a different point of view and use Heegaard Floer homology to study Miyazaki knots.  We prove the vanishing result in Theorem \ref{main2}.  Discussion in this section will constitute a rather independent unit of the paper.

\vspace{5pt}\noindent{\bf Acknowledgements.} The authors thank Tetsuya Abe, Jae Choon Cha, and Stefan Friedl for helpful discussions and suggestions. The authors are particularly indebted to Jae Choon Cha for conversations that greatly inspired this collaboration.  Part of this work was done while the first author was
visiting The Chinese University of Hong Kong\@.  The first author thanks The Chinese University of Hong Kong for its generous
hospitality and support. The second author was partially supported by grants from the Research Grants Council of the Hong Kong Special Administrative Region, China (Project No.\ 24300714 and 14301215).

\section{Preliminaries on $-$amphicheiral knots}\label{section:amphicheiral}
In this section, we survey results of $-$amphicheiral knots that will be key ingredient for proving Theorem \ref{main1}. First, we define $-$amphicheiral knots and strongly $-$amphicheiral knots.
\begin{definition}A knot $K$ in $S^3$ is \emph{$-$amphicheiral} if there exists an orientation reversing homeomorphism $f\colon (S^3,K)\to (S^3,K)$ such that $f(K)$ is $K$ with the reversed orientation. A knot $K$ in $S^3$ is \emph{strongly $-$amphicheiral} if we can choose $f$ to be an involution.
\end{definition}

More generally, we say that a link $(S^3, L=K_1\sqcup\cdots \sqcup K_n)$ \emph{has symmetry} $(\alpha, \epsilon_1, \cdots ,\epsilon_n)$ if there exists a self-homeomorphism $f$ of $S^3$ of class $\alpha$ that restricts to a self-homeomorphism of each component $K_i$ of class $\epsilon_i$ for each $i$.  Here, $\alpha$ takes the value $\pm 1$ or $J_{\pm}$, which stands for orientation preserving/reversing homeomorphisms or involutions of $S^3$, respectively; and $\epsilon_i=\pm 1$ depending on whether $f|_{K_i}$ preserves or reverses homeomorphisms of $K_i$.  In particular, a knot is $(S^3, K)$ is $-$amphicheiral if it has symmetry $(-1, -1)$, and it is strongly $-$amphicheiral if it has symmetry $(J_-, -1)$. 

Every knot is either hyperbolic, a torus, or a satellite knot.  It is not hard to see that non-trivial torus knots are not $-$amphicheiral. By Lemma \ref{lemma:Kawauchi} below, the notions of $-$amphicheiral and strongly $-$amphicheiral are equivalent for hyperbolic knots.  Thus, we shall primarily focus on $-$amphicheiral satellite knots.  Let $K=P(J)$ denote a satellite knot $K$ with pattern $(S^1\times D^2, P)$ and companion $J$. It turns out that the symmetry of $K$ is almost completely determined by symmetries of $P$ and $J$.

Following the notation of Hartley \cite{Hartley:1980-2}, we say a knot in a solid torus $(S^1\times D^2, K)$ has symmetry $([\alpha, \epsilon_1], \epsilon_2)$ if there exists a self-homeomorphism of the solid torus of class $\alpha$ that maps the longitude class $[\lambda]$ to $\epsilon_1[\lambda]$ and restricts to a self-homeomorphism of $K$ of the class $\epsilon_2$.  As before, $\alpha$ takes the value $\pm 1$ or $J_{\pm}$ that stands for orientation preserving/reversing homeomorphisms or involutions of the solid torus $S^1\times D^2$, respectively; and $\epsilon_1$, $\epsilon_2$ take the value $\pm 1$.

We are now in a position to state the key lemma of \cite{Hartley:1980-2}, which relates the symmetry of a satellite knot $(S^3, P(J))$ with the symmetry of the companion $(S^3, J)$ and the pattern $(S^1\times D^2,P)$.

\begin{lemma}[{\cite[Theorem 4.1(1)]{Hartley:1980-2}}]\label{lemma:Hartley4.1} Suppose $P$ is a pattern and $J$ is a non-trivial prime knot and neither $J$ nor its mirror image is a companion of $P(U)$.  Let $\alpha=\pm 1$ or $J_\pm$ and $\epsilon=\pm 1$. Then, $(S^3, P(J))$ has symmetry $(\alpha, \epsilon)$ if and only if $(S^3, J)$ has symmetry $(\alpha, \epsilon_1)$ and $(S^1\times D^2, P)$ has symmetry $([\alpha, \epsilon_1], \epsilon)$ for some $\epsilon_1=\pm 1$.
\end{lemma}

Here $U$ denotes the unknot, so $P(U)$ is a knot in $S^3$ that has the same pattern as $P$.  In particular, $P\subset S^1\times D^2$ is called an \emph{unknotted pattern} if $P(U)$ is an unknot, and the mild technical condition of the lemma is satisfied for all unknotted patterns.  To prove the lemma, we deform a given homeomorphism $f$ of $(S^3,K)$ of symmetry $(\alpha, \epsilon)$ to an isotopic homeomorphism $f'$ which fixes $J$ and a tubular neighborhood $V$ of it.  Suppose $f'$ maps $[\lambda]$ to $\epsilon_1[\lambda]$, then the induced self-homeomorphism $f'|_V$ has the symmetry $([\alpha, \epsilon_1], \epsilon)$ on $(V,K)\cong (S^1\times D^2, P)$.  Since $f'$ also fixes $J$ which is in the same class of the longitude $\lambda$, the self homeomorphism $f'$ realizes the symmetry $(\alpha, \epsilon_1)$ for $(S^3, J)$ . The converse of the lemma can be proved in a similar manner, and we refer the details to Hartley's original manuscript \cite{Hartley:1980-2}.

\begin{example}\label{example1}
As an application of Lemma \ref{lemma:Hartley4.1}, we prove that the satellite knot $K=P(J)$ is strongly $-$amphicheiral, where $J$ is the figure-eight knot and $P$ is the pattern as in Figure \ref{Livingston}.

\begin{figure}[h!]
\includegraphics[scale=0.4]{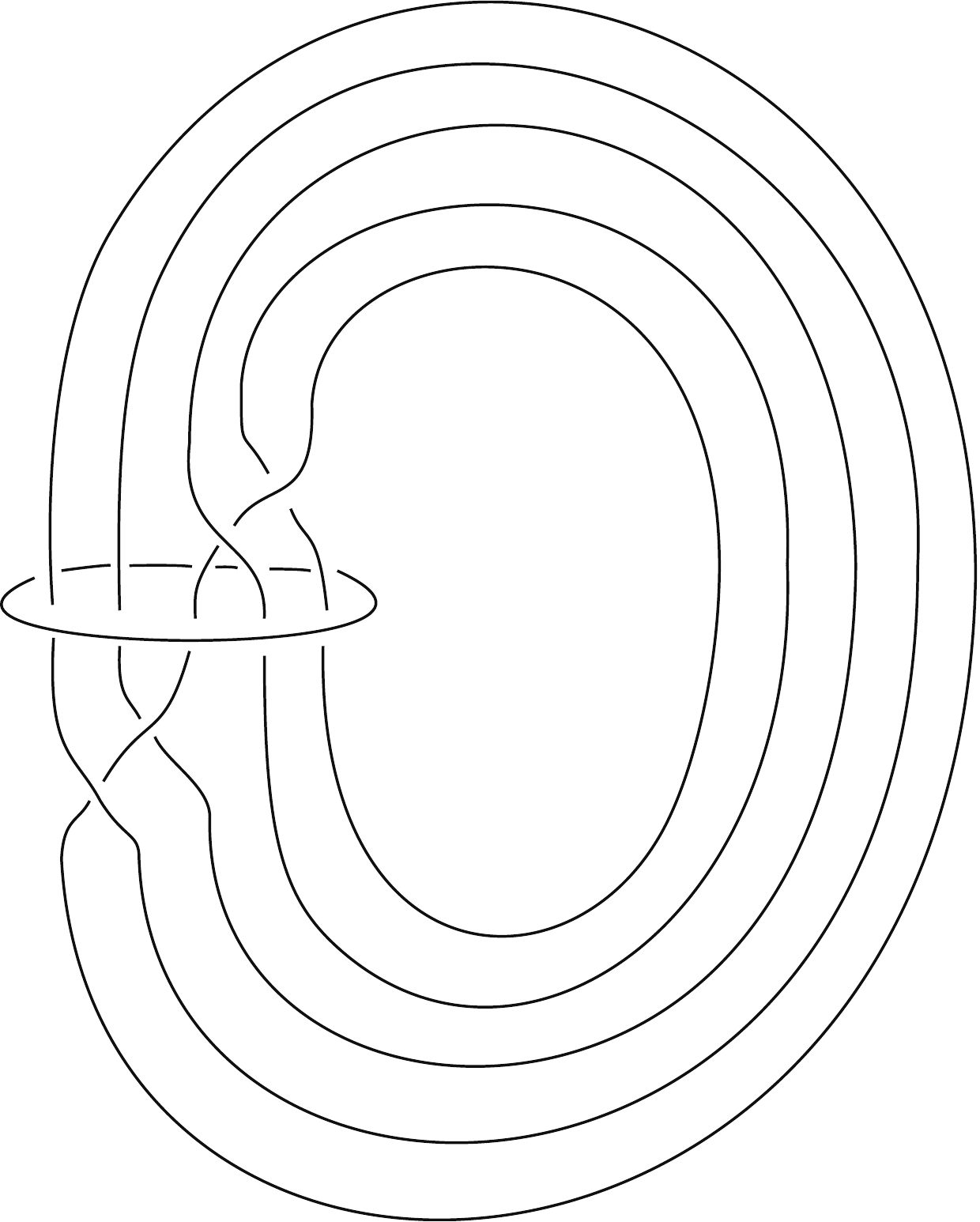}
\caption{$P$ is the closure of a braid of the form $\sigma_4\sigma_3^{-1}\sigma_2\sigma_1^{-1}$. It is an unknotted pattern and has symmetry $([J_-, -1], -1)$.  }
\label{Livingston}
\end{figure}

Since $(S^3, J)$ has symmetry $(J_-, -1)$, it suffices to show that $(S^1\times D^2, P)$ has symmetry $([J_-, -1], -1)$.  If we parameterize the meridional disk and the longitude of the solid torus by $z$ and $\theta$, respectively, then the map $f: (z, \theta) \rightarrow (-z, -\theta)$ is the desired orientation reversing involution of $(S^1\times D^2)$ that maps $P$ to an orientation reversing copy of itself. It follows from Lemma \ref{lemma:Hartley4.1} that $(S^3,K)$ has symmetry $(J_-, -1)$.
\end{example}

In \cite[Theorem 4.1(3)]{Hartley:1980-2}, Hartley relates the symmetry of a pattern $(S^1\times D^2,P)$ with the symmetry of \emph{the associated link} $(S^3,\mu_P,P(U))$ of $P$ where $\mu_P$ is the meridian of $P$.  Most relevant to our purpose is the following special case.

\begin{lemma}[{\cite[Theorem 4.1(3)]{Hartley:1980-2}}]\label{lemma:Hartley4.1.3} Suppose $P$ is a pattern and $\epsilon_i=\pm 1$ for $i=0,1,2$. Then, $(S^1\times D^2,P)$ has symmetry $([\epsilon_0,\epsilon_1],\epsilon_2)$ if and only if its associated link $(S^3,\mu_P,P(U))$ has symmetry $(\epsilon_0,\epsilon_0\epsilon_1,\epsilon_2)$. Also, $(S^1\times D^2, P)$ has symmetry $([J_-, \epsilon_1], \epsilon_2)$ if and only if $(S^3,\mu_P,P(U))$ has symmetry $(J_-,-\epsilon_1,\epsilon_2)$
\end{lemma}

As observed by Hartley \cite[page 106]{Hartley:1980-1}, if a 2-component link $(S^3,K_1,K_2)$ has symmetry $(-1,\epsilon_1,\epsilon_2)$ with $\epsilon_1\epsilon_2=1$, then the linking number $\textrm{link}(K_1,K_2)=-\textrm{link}(K_1,K_2)=0$. This gives a simple criterion on symmetries of $(S^1\times D^2, P)$ with non-trivial winding number $w(P)$ as follows.

\begin{lemma}\label{remark:Hartley4.1.3}
Suppose $P$ is a pattern with non-zero winding number. If $(S^1\times D^2, P)$ has symmetry $([-1, \epsilon_1], \epsilon_2)$, then $\epsilon_1\epsilon_2=1$.
\end{lemma}

\begin{proof}
By Lemma \ref{lemma:Hartley4.1.3}, $(S^1\times D^2, P)$ has symmetry $([-1, \epsilon_1], \epsilon_2)$ if and only if its associated link $(S^3,\mu_P,P(U))$ has symmetry $(-1,-\epsilon_1,\epsilon_2)$.  As $\textrm{link}(\mu_P,P(U))$ is equal to the winding number of $P$, the statement readily follows from Hartley's observation above.
\end{proof}

\section{Miyazaki knots are rationally slice}
The goal of this section is to prove Theorem \ref{main1}.  We begin with a lemma that essentially follows from Kawauchi \cite{Kawauchi:1979-1,Kawauchi:2009-1}.  Recall that $\Z[\frac{1}{2}]$ is the subring of $\Q$ generated by $\frac{1}{2}$.
\begin{lemma}\label{lemma:Kawauchi}Suppose that $K$ is a knot in $S^3$.
\begin{enumerate}
\item If $K$ is $-$amphicheiral and hyperbolic, then $K$ is strongly $-$amphicheiral.
\item If $K$ is strongly $-$amphicheiral, then $K$ is $\Z[\frac{1}{2}]$-slice and hence $K$ is rationally slice.
\end{enumerate}
\end{lemma}
\begin{proof}[Proof of Lemma \ref{lemma:Kawauchi}] (1) is a special case of \cite[Lemma~1]{Kawauchi:1979-1}. We sketch below a proof of (2), which is adapted from \cite[Lemma 2.3]{Kawauchi:2009-1}.

Suppose $K$ is strongly $-$amphicheiral. That is, there is an orientation reversing involution $\tau\colon S^3\to S^3$ such that $\tau(K)=K$ and $\operatorname{Fix}(\tau)=S^0$. Let $M$ be the 0-surgery manifold of $S^3$ along $K$. Then, $\tau$ naturally extends to the orientation reversing involution $\tau_M\colon M\to M$.

Let $M_\tau$ be the quotient space $M/{\sim}$ where $x\sim\tau_M(x)$ for all $x\in M$. Then, $M_\tau$ is a connected and non-orientable 3-manifold. Since $\tau_M$ is an orientation reversing involution, the quotient map $\pi\colon M\to M_\tau$ is the orientation double cover. Let $Z$ be the twisted $I$-bundle over $M_\tau$. That is, $Z$ is the mapping cylinder of $\pi\colon M\to M_\tau$ and  $\d Z= M$. Note that $M_\tau$ is a deformation retract of $Z$ and hence $H_*(M_\tau;\Z)\cong H_*(Z;\Z)$.

In the proof of \cite[Lemma 2.3]{Kawauchi:2009-1}, it is proved that $H_*(Z;\Q)=H_*(S^1;\Q)$ but we will prove a stronger statement that $H_*(Z;\Z[\frac{1}{2}])\cong H_*(S^1;\Z[\frac{1}{2}])$. Since $M_\tau$ is a connected and non-orientable 3-manifold,
\[H_i(Z;\Z)\cong H_i(M_\tau;\Z)=\left\{ \begin{array}{rl}
 \Z&\mbox{if $i=0$} \\
  0 &\mbox{if $i=3,4$.}
       \end{array} \right.
\]
In \cite[Lemma 2.3]{Kawauchi:2009-1}, it is proved that $H_1(Z,M;\Z)=\Z_2$ and $H_1(Z;\Q)=\Q$. As $M$ is the 0-surgery manifold of $S^3$ along $K$, $H_1(M;\Z)=\Z$. From the homology long exact sequence of a pair $(Z,M)$, we have an exact sequence
\[\Z\to H_1(Z;\Z)\to \Z_2\to 0\]
which gives $H_1(Z;\Z)=\Z$ or $\Z\oplus \Z_2$.

Since $\Z[\frac{1}{2}]$ is a torsion free, abelian group, $\Z[\frac{1}{2}]$ is a flat $\Z$-module. Hence, $H_*(-;\Z[\frac{1}{2}])\cong H_*(-;\Z)\otimes\Z[\frac{1}{2}]$ by universal coefficient theorem. It follows that
\[ H_i(Z;\Z[\tfrac{1}{2}])\cong H_*(Z;\Z)\otimes\Z[\tfrac{1}{2}]=\left\{ \begin{array}{cl}
 \Z[\tfrac{1}{2}]&\mbox{if $i=0,1$} \\
  0 &\mbox{if $i=3,4$.}
       \end{array} \right.
\]
Now, we prove that $H_2(Z;\Z[\frac{1}{2}])=0$. Let $p$ be an odd prime. From the universal coefficient theorem,
\[H_i(Z;\Z_p)=\left\{ \begin{array}{cl}
 \Z_p&\mbox{if $i=0,1$} \\
  0 &\mbox{if $i=3,4$.}
       \end{array} \right.
\]
As mentioned above, $H_*(Z;\Q)\cong H_*(S^1;\Q)$, so the Euler characteristic $\chi(Z)$ of $Z$ is zero.  Therefore,
\[ 0=\chi(Z)=\sum_{i=0}^4 (-1)^i\dim_{\Z_p}H_i(Z;\Z_p)=\dim_{\Z_p} H_2(Z;\Z_p)\]
which implies that $H_2(Z;\Z_p)=0$ for all odd prime $p$.

Recall that $H_2(Z;\Q)\cong H_2(S^1;\Q)=0$. This implies that $H_2(Z;\Z)$ is a torsion abelian group. For an odd prime $p$, by universal coefficient theorem, $0=H_2(Z;\Z_p)\cong H_2(Z;\Z)\otimes \Z_p$. Form this, the order of $H_2(Z;\Z)$ is a power of $2$ and $H_2(Z;\Z[\frac{1}{2}])=H_2(Z;\Z)\otimes\Z[\frac{1}{2}]=0$.

In summary, we observed that if $K$ is a strongly $-$amphicheiral knot in $S^3$, then the 0-surgery manifold $M$ bounds a 4-manifold $Z$ such that $H_*(Z;\Z[\frac{1}{2}])\cong H_*(S^1;\Z[\frac{1}{2}])$. It is well-known that the existence of such $Z$ is equivalent to the condition that $K$ is $\Z[\frac{1}{2}]$-slice (see \cite[Proposition 1.5]{Cochran-Franklin-Hedden-Horn:2013-1}). From the universal coefficient theorem, it is easy to see that $K$ is also rationally slice.\end{proof}

Our next lemma concerns Miyazaki knots that are satellite.  Noticeably, we find that the companion knot $J$ of the satellite Miyazaki $K=P(J)$ is also Miyazaki.

\begin{lemma}\label{lemma:induction}If a satellite knot $K=P(J)$ is Miyazaki, then $P$ is an unknotted pattern, $J$ is Miyazaki, and $g(J)<g(K)$.
\end{lemma}

Assuming Lemma \ref{lemma:induction}, we prove Theorem \ref{main1}.

\begin{proof}[Proof of Theorem \ref{main1}]Suppose that $K$ is a non-trivial Miyazaki knot. That is, $K$ is fibered, $-$amphicheiral with $\Delta_K(t)$ irreducible. We will use an induction on $g(K)$. Suppose that $g(K)=1$. Then, by the classification of genus 1 fibered knots, $K$ is the figure-eight knot which is hyperbolic. By Lemma \ref{lemma:Kawauchi}, $K$ satisfies the conclusion.

Suppose $g(K)>1$. As an induction hypothesis, we assume that if $K'$ is Miyazaki and $g(K')<g(K)$, then $K'$ is $\Z[\frac{1}{2}]$-slice.

Every knot is precisely one of the following: hyperbolic, a torus knot, or a satellite knot. Note that any non-trivial torus knot is not $-$amphicheiral.  In light of  Lemma \ref{lemma:Kawauchi},  we only need to prove the theorem for $K$ a non-trivial satellite knot.

By Lemma \ref{lemma:induction} and the induction hypothesis, we assume $K=P(J)$ such that $P$ is an unknotted pattern, and $J$ bounds a slice disk $\Delta$ in a $\Z[\frac{1}{2}]$-homology $4$-ball $V$. The subsequent construction of a slice disk for $K$ is standard:  Suppose $P\subset S^1\times D^2 \subset \partial (D^2\times D^2)$ bounds a disk $D\subset D^2\times D^2$.  As the tubular neighborhood $\nu(\Delta)\cong\Delta\times D^2$ is diffeomorphic to $D^2\times D^2$, the image of the disk $D\subset D^2\times D^2 \cong \nu(\Delta)\subset V$ under the above diffeomorphism is then the desired slice disk for $K=P(J)\subset V$. This finishes the proof.
\end{proof}

The proof of Lemma \ref{lemma:induction} is the most technical part of the paper. In the course of the proof, we will make use of results of $-$amphicheiral, satellite knots that we discussed earlier in Section 2, and the following criterion for fibered, satellite knots (see, e.g., \cite[Theorem~1]{Hirasawa-Murasugi-Silver:2008-1}).

\begin{lemma}[Criterion for fibered, satellite knots]\label{lemma:fiberedsatellite}
A satellite knot $K=P(J)$ is fibered if and only if both the companion knot $J$ and the pattern $(S^1\times D^2, P)$ are fibered.
\end{lemma}

\begin{proof}[Proof of Lemma \ref{lemma:induction}] Suppose $K$ is satellite and Miyazaki. That is, $K$ is satellite, fibered, and $-$amphicheiral with $\Delta_K(t)$ irreducible. We first observe that $K$ is a prime knot. Suppose $K=K_1\#K_2$. Since $K$ is fibered, $K_1$ and $K_2$ are also fibered. Then, $\deg\Delta_{K_i}(t)=2g(K_i)$ and $\Delta_K(t)=\Delta_{K_1}(t)\Delta_{K_2}(t)$. Since $\Delta_K(t)$ is irreducible, after a possible reordering, $\Delta_{K_1}(t)=1$ which implies that $g(K_1)=0$ and $K$ is prime. Our argument shows that a fibered knot with irreducible Alexander polynomial is prime.



Since $K$ is a satellite knot, we can write $K=P(J)$ such that $J$ is a non-trivial knot and the pattern $P$ is not isotopic to the core of the solid torus. Since $K$ is fibered, Lemma \ref{lemma:fiberedsatellite} implies that $(S^1\times D^2,P)$ and $J$ are fibered. Moreover, the winding number $w$ of $P$ must be non-zero, according to the proof of \cite[Theorem 1]{Hirasawa-Murasugi-Silver:2008-1}.

We have a cabling formula $\Delta_K(t)=\Delta_P(t)\Delta_J(t^{w})$ where $w$ is the winding number. Since $P$ and $J$ are fibered, $\deg\Delta_P(t)=2g(P)$ and $\deg \Delta_J(t)=2g(J)$. From the irreducibility of $\Delta_K(t)$ and $\Delta_J(t^w)\neq 1$, we conclude that $\Delta_P(t)=1$ and hence $P$ is an unknotted pattern.

After a possible simultaneous change of the orientations of $P$ and $K$, we can assume that the winding number $w$ is positive. By \cite[Corollary 1]{Hirasawa-Murasugi-Silver:2008-1}, if $P$ is a winding number 1 unknotted pattern such that $(S^1\times D^2, P)$ is fibered, then $P$ is isotopic to the core of $S^1\times D^2$. Therefore, $|w|\geq 2$. From the cabling formula and fiberedness, $g(K)=|w|g(J)>g(J)$. Since $\Delta_K(t)$ is irreducible, $\Delta_J(t^w)=\Delta_K(t)$ implies that $\Delta_J(t)$ is also irreducible.

It remains to prove that $J$ is $-$amphicheiral. For this purpose, we apply Lemma \ref{lemma:Hartley4.1}. We first check that $J$ and $P$ satisfies the hypothesis of Lemma \ref{lemma:Hartley4.1}. In the beginning of the proof, we proved that a fibered knot with irreducible Alexander polynomial is prime. Therefore, $J$ is also prime. Recall that we are assuming $J$ is non-trivial. We proved that $P(U)$ is the unknot. Therefore, neither $J$ nor its mirror image is a companion of $P(U)$.

Recall that $K$ is $-$amphicheiral if and only if $(S^3,K)$ has symmetry $(-1,-1)$. Lemma \ref{lemma:Hartley4.1} then implies that $(S^3, J)$ has symmetry $(-1, \epsilon_1)$ and $(S^1\times D^2, P)$ has symmetry $([-1, \epsilon_1],-1)$ for some $\epsilon_1=\pm 1$. Since $P$ is a pattern with non-zero winding number, Lemma \ref{remark:Hartley4.1.3} gives that $\epsilon_1=-1$. This shows that $(S^3,J)$ has symmetry $(-1,-1)$ and hence $J$ is $-$amphicheiral. This completes the proof.
\end{proof}

We finish this section by recalling the following results of \cite{Miyazaki:1994-1}, which will then be used to prove Corollary \ref{corollary:slice-ribbon}.


\begin{lemma}[{\cite{Miyazaki:1994-1}}]\label{lemma:Miyazaki}Suppose that $J$ is a non-trivial fibered knot and $p>1$.
\begin{enumerate}
\item If $K=J_{p,q}$ is the $(p,q)$-cable of $J$ and $\underbrace{K\#\cdots\#K}_\text{n}$ is homotopy ribbon for some $n$, then so is $K$.
\item If $J$ is Miyazaki and $\#_{n=1}^ma_nJ_{2n,1}$ is homotopy ribbon, then $a_n=0$ for all $n$.
\end{enumerate}
\end{lemma}
\begin{proof}
(1) This is a restatement of \cite[Theorem 8.5.1]{Miyazaki:1994-1}.

(2) We show that it is a consequence of Theorem 8.6 of \cite{Miyazaki:1994-1}. Indeed, define $\mathcal{K}$ to be the class of knots that consists of all iterated cables of all prime fibered knots $J$ with the condition that there is no non-trivial Laurent polynomial $f(t)f(t^{-1})$ divides $\Delta_J(t)$. 
In \cite[Theorem 8.6]{Miyazaki:1994-1}, Miyazaki proved that if $K_1,\ldots, K_m$ are knots in $\mathcal{K}$ such that $K_1\#\cdots \#K_m$ are homotopy ribbon, then $m=2l$ and $K_{2n-1}=-K_{2n}$ for $n=1,\ldots,l$ after a possible relabeling of indices.

Suppose $J$ is Miyazaki and $\#_{n=1}^m a_nJ_{2n,1}$ is homotopy ribbon.  We further assume that $a_k$ is non-zero for some $k$. Since $J$ is fibered, $g(J_{2n,1})=2ng(J)$ and hence $J_{2i,1}$ is not equal to $-J_{2j,1}$ if $i\neq j$.  Note that $J_{n,1}$ is in $\mathcal{K}$ for all $n$ if $J$ is Miyazaki, so \cite[Theorem 8.6]{Miyazaki:1994-1} shows that $a_n$ is even for every $n$ and $J_{2k,1}$ is $-$amphicheiral. Thus, $J_{2k,1}\#J_{2k,1}=J_{2k,1}\#-J_{2k,1}$ is ribbon. By (1), this implies that $J_{2k,1}$ is homotopy ribbon, and this contradicts to Proposition \ref{example:Miyazaki} where it is shown that $J_{2n,1}$ is not homotopy ribbon for a Miyazaki knot $J$.
\end{proof}

\begin{proof}[Proof of Corollary \ref{corollary:slice-ribbon}]Suppose $K$ is a Miyazaki knot and that $\#_{n=1}^m a_nK_{2n,1}$ is slice. Assuming the slice-ribbon conjecture, $\#_{n=1}^m a_n K_{2n,1}$ is (homotopy) ribbon. By Lemma \ref{lemma:Miyazaki}(2), $a_n=0$ for every $n$. It follows that $\{K_{2n,1}\}_{n=1}^\infty$ form a $\Z^\infty$-subgroup in $\mathcal{C}$. Since $K$ is Miyazaki, $K$ is rationally slice from Theorem \ref{main1}. This completes the proof that $\{K_{2n,1}\}$ generate a $\Z^\infty$-subgroup in $\Ker(\mathcal{C}\to \mathcal{C}_\Q)$ assuming the slice-ribbon conjecture.
\end{proof}

\section{Examples of Miyazaki knots}

Our next proposition exhibits an infinite family of satellite Miyazaki knots. In particular, we will see that the knot given in Example \ref{example1} is Miyazaki.  To the best of the authors' knowledge, there has been no construction of satellite Miyazaki knot before.

\begin{proposition}\label{SatelliteMiyazaki}
Suppose $J$ is the figure-eight knot, and $P_n$ is the closure of a $(2n+1)$-braid of the form $\prod_{i=0}^{2n} \sigma_{2n-i}^{(-1)^i}$. If $2n+1$ is a power of $5$, then the satellite knot $K=P_n(J)$ is Miyazaki.
\end{proposition}

\begin{proof}
Recall from Definition \ref{definition:Miyazaki}, we need to prove that $K$ is fibered, $-$amphicheiral, and $\Delta_K(t)$ is irreducible.

A braid $\beta$ is called \emph{homogeneous} if each standard braid generator $\sigma_i$ appears at least once in $\beta$ and the exponent on $\sigma_i$ has the same sign in each appearance in the braid word $\beta$ (for example, if $\sigma_i$ appears, then $\sigma_i^{-1}$ does not appear).  A theorem of Stalling \cite[Theorem 2]{Stallings:1978-1} says that $(S^1\times D^2,\hat{\beta})$ is fibered for any homogeneous braid $\beta$. Since $P_n$ is the closure of a homogeneous braid $\prod_{i=0}^{2n} \sigma_{2n-i}^{(-1)^i}$, $(S^1\times D^2,P_n)$ is fibered. As $J=4_1$ is fibered, we see that $K=P_n(J)$ is also fibered from Lemma \ref{lemma:fiberedsatellite}.

Just like Example \ref{example1}, we observe that $P_n$ is an unknotted pattern and $(S^1\times D^2, P_n)$ has symmetry $([J_-, -1]. -1)$.  Hence $P_n(J)$ has symmetry $(J_-, -1)$ by Lemma \ref{lemma:Hartley4.1}, which is strongly $-$amphicheiral.  Next, we apply the cabling formula of the Alexander polynomial $\Delta_K(t)=\Delta_{P_n}(t)\Delta_J(t^{2n+1})$ (the winding number is $2n+1$), so $\Delta_K(t)=t^{2(2n+1)}-3t^{2n+1}+1$.

Note that $\Delta_K(t)$ is irreducible if and only if $\Delta_K(t-1)$ is irreducible, and suppose
\begin{equation*}
\begin{split}
\Delta_K(t-1)&=(t-1)^{2(2n+1)}-3(t-1)^{2n+1}+1\\
&=((t-1)^{2n+1}+1)^2-5(t-1)^{2n+1}\\
&=t^{4n+2}+a_{4n+1}t^{4n+1}+\cdots +a_1t+5.
 \end{split}
\end{equation*}
We claim that each $a_i$ is a multiple of 5 when $2n+1$ is a power $5^m$ for some $m$.  This follows from the well-known fact that the binomial coefficient ${5^m}\choose{k}$ is divisible by 5 for all $0<k<5^m$, which then implies that the first term in the second line of the above identity $((t-1)^{2n+1}+1)^2$ is equal to $(t^{2n+1})^2$ as polynomials with coefficients modulo 5.

Finally, we apply the Eisenstein criterion with $p=5$ to prove the irreducibility of $\Delta_K(t-1)$, i.e., we check that $p$ divides each $a_i$, and $p^2$ does not divide $a_0=5$.  As $\Delta_K(t)$ is irreducible if and only if $\Delta_K(t-1)$ is irreducible, this completes the proof.
\end{proof}

\begin{remark}(1) There is an alternative way to see that $(S^1\times D^2,P_n)$ is fibered due to Morton \cite[page 102]{Morton:1978-1}. Since $P_n$ is an unknotted pattern, $S^3-\nu(P_n(U))$ is the solid torus $V$. By definition, $P_n$ is the closure of a braid $\prod_{i=0}^{2n} \sigma_{2n-i}^{(-1)^i}$ and the meridian $\mu_{P_n}$ of $P_n$ is the axis of the braid. Hence, $S^1\times D^2-\nu(P_n)\cong V-\nu(\mu_{P_n})$ is fibered over $S^1$ having a disk with $2n+1$ holes as fiber.\\
(2) We used Mathematica and checked that the polynomials $t^{2(2n+1)}-3t^{2n+1}+1$ are irreducible for all $n<100$ (and consequently, the knots $K=P_n(J)$ are Miyazaki).  We speculate that the technical assumption on $2n+1$ being a power of $5$ may be unnecessary.

\end{remark}

Note that the Miyazaki knots constructed in Proposition \ref{SatelliteMiyazaki} are strongly $-$amphicheiral.  In general, remember that Miyazaki knots are $-$amphicheiral, and hyperbolic $-$amphicheiral knots are strongly $-$amphicheiral.  We ask:

\begin{question}
Are Miyazaki knots always strongly $-$amphicheiral?
\end{question}

From the discussion in the previous section, we see that all Miyazaki knots can be obtained from hyperbolic ones via iterated satellite operations. This inspires us to look for an inductive approach, and we establish the following result along this direction.

\begin{proposition}
Suppose $K=P(J)$ is a Miyazaki knot with a hyperbolic companion $J$ and a pattern of winding number $3$.  Then $K$ is strongly $-$amphicheiral.
\end{proposition}

\begin{proof}
From Lemma \ref{lemma:induction}, we see that $J$ is also Miyazaki and $P$ is a fibered unknotted pattern.  Lemma \ref{lemma:Kawauchi} then implies that $J$ is strongly $-$amphicheiral.  In light of Lemma \ref{lemma:Hartley4.1}, it suffices to prove that $(S^1\times D^2, P)$ has symmetry $([J_-, -1], -1)$.

Note that $P$ is a fibered pattern of winding number $3$, and there are only three such patterns up to isotopy in $S^1\times D^2$, corresponding to the closure of conjugacy class of $3$-braids $\sigma_1\sigma_2$, $\sigma_1^{-1}\sigma_2^{-1}$ and $\sigma_1^{-1}\sigma_2$ \cite{Magnus-Peluso:1967-1}.  The first two patterns of $P$ give cable knots $K=J_{3,1}$ and $J_{3,-1}$, respectively.  We will prove that $K$ is not Miyazaki, and thus these two cases do not occur.  This leaves the third pattern as the only possibility.  For this case, we observe that the closure of the braid $\sigma_1^{-1}\sigma_2$ has indeed the symmetry $([J_-, -1], -1)$ (see the explanation in Example \ref{example1}), and this finishes the proof.

Let us prove the more general statement that a $(p,q)$-cable knot $K=J_{p,q}$ with $p>1$ is not Miyazaki.  We use proof by contradiction, and assume, to the contrary, that $K=J_{p,q}$ was $-$amphicheiral and fibered.  As $K\# K=K\# -K$ is ribbon,  Lemma \ref{lemma:Miyazaki} implies that $K$ is homotopy ribbon.  Since an $(m,1)$-cable of a homotopy ribbon, fibered knot is also homotopy ribbon \cite[page 2]{Miyazaki:1994-1}, we conclude that $K_{2,1}$ is homotopy ribbon.  But that contradicts to Proposition \ref{example:Miyazaki} where it is shown that $K_{2,1}$ is not homotopy ribbon when $K$ is Miyazaki.
\end{proof}

\section{$\nu^+$-invariants of satellites do not detect slice knots}
Throughout this section, the unknot is denoted by $U$. A pattern $Q\subset S^1\times D^2$ is called a \emph{slice pattern} if $Q(U)$ is slice.  The $\Q$-homology $0$-bipolar knot is defined in \cite[Definition 2.3]{Cha-Powell:2014-1} as a $\Q$-homology version of the notion of $0$-bipolar knots introduced in \cite{Cochran-Harvey-Horn:2013-1}.  For the reader's convenience, we recall the definition here.

\begin{definition}\label{Def:HBP}
A knot $K$ in $S^3$ is \emph{$\Q$-homology $0$-bipolar} if there exist pairs $(V_+,D_+)$, $(V_-,D_-)$ of compact smooth 4-manifold $V_\pm$ and a smoothly embedded disk $D_\pm$ in $V_\pm$ such that
\begin{enumerate}
\item  $\partial(V_\pm,D_\pm)=(S^3,K)$.
\item $H_1(V_\pm;\Q)=0$.
\item $V_{\pm}$ is $\pm$-definite. That is, $b^\pm_2(V_\pm)=b_2(V_\pm)$.
\item $[D_\pm, \d D_\pm]=0\in H_2(V_\pm,S^3;\Q)$.
\end{enumerate}
\end{definition}

We will use the following facts about $\Q$-homology $0$-bipolar knots.
\begin{itemize}
\item[(B1)]\label{item:RS-HBP} A rationally slice knot is $\Q$-homology $0$-bipolar (Compare Definition \ref{Def:RS} and \ref{Def:HBP})
\item[(B2)]\label{item:satellite-HBP} If $K$ is $\Q$-homology $0$-bipolar
  knot, then $Q(K)$ is $\Q$-homology $0$-bipolar for any slice pattern~$Q$ \cite[Theorem~2.6~(6)]{Cha-Powell:2014-1}.
\item[(B3)]\label{item:HBP-d}$V_0(K)=V_0(-K)=0$ if $K$ is $\Q$-homology
  $0$-bipolar \cite[Theorem~2.7]{Cha-Powell:2014-1}.
\end{itemize}
 We remark that (B2) and (B3) are mild generalizations of
\cite[Propositions~3.3 and~1.2]{Cochran-Harvey-Horn:2013-1}.  The slice disk for $Q(K)$ in (B2) is constructed in the same way as that in the proof of Theorem \ref{main1}. (B3) was originally stated in terms of the correction terms $d(S^3_1(K))=d(S^3_1(-K))=0$, and we refer to \cite{Hom-Wu:2014-1} for the equivalence of these two identities. The function $V_0(K)$, and more generally, the sequence $\{V_k(K)\}$ of a knot $K$ will be defined shortly.

Next, we sketch the construction of $\nu^+$ and relevant background of Heegaard Floer homology.  For a knot $K\subset S^3$, the Heegaard Floer knot complex $CFK^\infty(K)$ is a doubly filtered complex with a $U$-action that lowers each filtration by one.  Let the quotient complexes $A^+_k = C\{ \max \{i, j-k\}  \geq 0\}$ and $B^+ = C\{ i \geq 0\}$,
where $i$ and $j$ refer to the two filtrations.  Associated to each $k$, there is a graded, module map
$$v^+_k: A^+_k \rightarrow B^+.$$  Define $V_k$ be the $U$-exponent of $v_k^+$ at sufficiently high gradings.  This sequence of $\{V_k\}$ is non-increasing, i.e., $V_k\geq V_{k+1}$, and stabilizes at 0 for large $k$.  The $\nu^+$ invariant is just the minimum $k$ for which $V_k=0$.  We list some properties of $\nu^+$ below.
\begin{itemize}
\item[(N1)] It is a concordance invariant, taking nonnegative integer value $$\nu^+(K)=\min\{k\in \Z\,|\, V_k=0\}\geq 0$$

\item[(N2)] It is bounded above by the four-ball genus $$\nu^+(K)\leq g_4(K)$$

\item[(N3)] It is bounded below by Ozsv\'ath-Stipsicz-Szab\'o's one-parameter family of concordance invariants \cite{Ozsvath-Stipsicz-Szabo:2014-1}
$$|\Upsilon_K(t)|\leq t \max(\nu^+(K),\nu^+(-K))$$
\end{itemize}



Following the same argument of \cite[Theorem B]{Cha-Kim:2015-1}, we give a proof of Theorem \ref{main2}.  Note that $\nu^+$ is not a concordance homomorphism, so part (2) of the theorem does not follow from part (1) immediately.

\begin{proof}[Proof of Theorem \ref{main2}]

(1) Note that $P\# -P(U)$ is a slice pattern for any pattern $P$. Hence by (B2), $P(K)\#-P(U)$ is $\Q$-homology 0-bipolar if $K$ is $\Q$-homology 0-bipolar. From (B3) and (N1), we conclude that
\[\nu^+(P(K)\#-P(U))=\nu^+(-P(K)\#P(U))=0.\]

(2) It is known that $\nu^+$ is sub-additive under connected sum \cite[Theorem 1.4]{Bodnar-Celoria-Golla:2015-1}. From (1) and concordance invariance of $\nu^+$, we have an inequality which holds for any $P$:
\[ \nu^+(P(U))=\nu^+(P(K)\#-P(K)\#P(U))\leq \nu^+(P(K))+\nu^+(-P(K)\#P(U))=\nu^+(P(K)).\]
The proof of $\nu^+(P(K))\leq \nu^+(P(U))$ is similar.
\end{proof}

\begin{corollary}\label{theorem:satellite}If $K$ is a $\Q$-homology $0$-bipolar knot, then
\begin{enumerate}
\item There is a filtered chain homotopy equivalence $CFK^\infty(P(K))\oplus A_1\simeq CFK^\infty(P(U))\oplus A_2$ of chain complexes where $A_1$, $A_2$ are acyclic.
\item $V_k(P(K)\#-P(U))=V_k(-P(K)\#P(U))=0$ for all $k\geq 0$.
\item $V_k(P(K))=V_k(P(U))$ for all $k\geq 0$.
\end{enumerate}
\end{corollary}

\begin{proof}

(1) and (2) are equivalent to Theorem \ref{main2}(1) by \cite[Theorem 1]{Hom:2015-2} and (N1), respectively. (3) follows from (1) and the fact that $V_k$ is independent of the acyclic summand.
\end{proof}

Using (N3), we show that the Ozsv\'ath-Stipsicz-Szab\'o $\Upsilon$ invariant of satellites does not detect slice knots either.
\begin{corollary}\label{Upsilon}
If $K$ is a $\Q$-homology $0$-bipolar knot, then
$$\Upsilon_{P(K)}=\Upsilon_{P(U)}$$ for all patterns $P$.
\end{corollary}

\begin{proof}
Since $\Upsilon$ is a concordance homomorphism, $$\Upsilon_{P(K)}-\Upsilon_{P(U)}=\Upsilon_{P(K)\#-P(U)}=0$$ by (N3) and Theorem \ref{main2}(1).
\end{proof}

We finish off our discuss with a question that is motivated by \cite{Cha-Kim:2015-1} and this paper.  From \cite[Proposition 5.1]{Hom:2014-1} and \cite[Theorem 3.1]{Cha-Kim:2015-1}, we know that the following conditions are equivalent:
  \begin{enumerate}
  \item Two knots $K$ and $K'$ are $\epsilon$-equivalent. That is, $\epsilon(K\#-K')=0$
  \item $\tau(P(K))=\tau(P(K'))$ for any pattern $P$.
  \item $\epsilon(P(K)\#-P(K'))=0$ for any pattern $P$.
  \item $\epsilon(P(K))=\epsilon(P(K'))$ for any pattern $P$.
  \end{enumerate}
In particular, any two $\Q$-homology 0-bipolar knots $K$ and $K'$ are $\epsilon$-equivalent \cite{Cha-Kim:2015-1}.  These two knots satisfies, in addition, $\Upsilon_{P(K)}=\Upsilon_{P(K')}(=\Upsilon_{P(U)})$ from Corollary \ref{Upsilon}.  In general, we ask:
\begin{question}Is $\Upsilon_{P(K)\#-P(K')}=0$ for all $\epsilon$-equivalent knots $K$ and $K'$ and patterns $P$?
\end{question}
Suppose there are $\epsilon$-equivalent knots $K$, $K'$, and a pattern $P$ such that $\Upsilon_{P(K)\#-P(K')}\neq 0$, then the knot $J=P(K)\#-P(K')$ will give the first example of knot with $\epsilon(J)=0$ and $\Upsilon_J\neq 0$. Previously, only a doubly-filtered complex $C$ was known to satisfies $\epsilon(C)=0$ and $\Upsilon_C\neq0$ \cite[Figure 6]{Ozsvath-Stipsicz-Szabo:2014-1}, but it is unclear whether such a complex can be realized as the knot Floer complex of some concrete knots. 

\bibliographystyle{amsalpha}
\renewcommand{\MR}[1]{}
\bibliography{research}
\end{document}